\documentclass[11pt]{amsart}

\usepackage[utf8]{inputenc}
\usepackage[english]{babel}
\usepackage{pdfpages}
\usepackage[top=3.5cm, bottom=3.5cm, left=2.5cm, right=2.5cm]{geometry}
\usepackage{graphicx}
\usepackage{amsmath,amsthm,amssymb}
\usepackage{wasysym}
\usepackage{microtype}
\usepackage[colorlinks=false, pdfborder={0 0 0}]{hyperref}
\usepackage[capitalise]{cleveref}

\title{The asymptotic dimension of box spaces of virtually nilpotent groups}
\author{Thiebout Delabie and Matthew C. H. Tointon}
\date{}
\thanks{The authors are supported by grant FN 200021\_163417/1 of the Swiss National Fund for scientific research.}

\newtheorem{theorem}{Theorem}[section]
\newtheorem{lemma}[theorem]{Lemma}
\newtheorem{proposition}[theorem]{Proposition}
\newtheorem{corollary}[theorem]{Corollary}

\theoremstyle{definition}
\newtheorem{definition}[theorem]{Definition}

\theoremstyle{remark}
\newtheorem*{remark*}{Remark}

\newcommand{\Cay}{\operatorname{Cay}}
\newcommand{\diam}{\operatorname{diam}}

\newcommand{\asdim}{\operatorname{asdim}}

\newcommand{\N}{\mathbb{N}}
\newcommand{\Z}{\mathbb{Z}}

\newcommand{\unif}{\text{\textup{unif}}}

\numberwithin{equation}{section}

\begin{document}

\maketitle

\begin{abstract}
We show that every box space of a virtually nilpotent group has asymptotic dimension equal to the Hirsch length of that group.
\end{abstract}

\section{Introduction}
The principal objects of study of this note are so-called \emph{box spaces}. These are metric spaces formed by stitching together certain finite quotients of residually finite groups, and have been the subject of much recent research, not least thanks to their utility in constructing metric spaces with unusual properties. For example, box spaces can be used to construct expanders \cite{Mar}, as well as to construct groups without property $A$ \cite{Gro,AD,Osa}.

Of particular interest to us will be the so-called \emph{asymptotic dimension} of a box space. Defined by Gromov \cite{AsdimGro}, the asymptotic dimension is a coarse version of the topological dimension. There has been a fair amount of recent work on computing the asymptotic dimension of various box spaces, and this note represents a contribution to that body of work.

Given a finitely generated residually finite infinite group $G$, a \emph{filtration} of $G$ is a sequence $(N_n)_{n=1}^\infty$ of nested, normal, finite-index subgroups $N_n$ of $G$ such that $\bigcap_n N_n=\{e\}$.
\begin{definition}[box space]
The \emph{box space} $\Box_{(N_i)}G$ of a finitely generated residually finite infinite group $G$ with respect to a filtration $(N_n)$ of $G$ and a finite generating set $S$ of $G$ is the metric space on the disjoint union of the quotients $(G/N_n)_n$ in which the metric on each component $G/N_n$ is the Cayley-graph metric induced by $S$, and the distance between $x\in G/N_m$ and $y\in G/N_n$ with $m\ne n$ is defined to be the sum of the diameters of $G/N_m$ and $G/N_n$.
\end{definition}

Now let $X$ be a metric space. Given $R>0$, a family $\mathcal U$ of subsets of $X$ is said to be \emph{$R$-disjoint} if the distance between every pair of distinct sets in $\mathcal U$ is at least $R$. Given $R,S>0$, the \emph{$(R,S)$-dimension} $(R,S)\text{-}\dim(X)$ of $X$ is defined to be the least $n\in\Z$ such that there exist families $\mathcal{U}_0,\mathcal{U}_1,\ldots,\mathcal{U}_n$ of subsets of $X$ such that
\begin{itemize}
\item $\bigcup_{j=0}^n \mathcal{U}_j$ covers $X$,
\item $\mathcal{U}_j$ is $R$-disjoint for every $j \in \{0,1,...,n\}$, and
\item $\diam(U) \le S$ for every $U \in \mathcal{U}_j$ and $j \in \{0,1,...,n\}$.
\end{itemize}
\begin{definition}[asymptotic dimension]
The \emph{asymptotic dimension} $\asdim X$ of a metric space $X$ is defined to be the least $n\in\Z$ such that for every $R>0$ there exists $S>0$ such that $(R,S)\text{-}\dim(X)\le n$.
\end{definition}

It is known that a virtually polycyclic group with a Cayley-graph metric has finite asymptotic dimension. Indeed, given a virtually polycyclic group $G$ we write $h(G)$ for the \emph{Hirsch length} of $G$, which is to say the number of infinite factors in a normal polycyclic series of a finite-index polycyclic subgroup of $G$. Dranishnikov and Smith \cite[Theorem 3.5]{DSHirsh} show that if $G$ is residually finite and virtually polycyclic then
\begin{equation}\label{eq:hirsch}
\asdim G=h(G).
\end{equation}

It is not unreasonable to expect that \emph{box spaces} of virtually polycyclic groups should also have finite asymptotic dimension, and indeed there have been some results in this direction. For example, Szab\'o, Wu and Zacharias \cite{Gabor} show that every finitely generated virtually nilpotent group has \emph{some} box space with finite asymptotic dimension. Finn-Sell and Wu \cite{FinnWu} show moreover that for certain box spaces of virtually polycyclic groups the asymptotic dimension of the box space is, like that of the group itself, equal to the Hirsch length of the group. These results are not ideal in their current form, as they rely on certain subgroup inclusions inducing coarse embeddings of the corresponding box spaces, a fact that doesn’t hold in general due to \cite[Theorem 4.9]{Fund}.


The main purpose of the present note is to clarify the situation and strengthen these results in the case of virtually nilpotent groups. Indeed, we show that if $G$ is a finitely generated virtually nilpotent group then in fact \emph{every} box space of $G$ has asymptotic dimension equal to the Hirsch length of $G$, as follows.

\begin{theorem}\label{thm:hirsch}
Let $G$ be a finitely generated residually finite virtually nilpotent group and let $(N_n)_n$ be a filtration of $G$. Then $\asdim \Box_{(N_n)}G = h(G)$.
\end{theorem}

The note is organised as follows. In \cref{sec:background} we present some basic facts about box spaces and about asymptotic dimension; in \cref{sec:arb} we compute the asymptotic dimension of certain box spaces in terms of the asymptotic dimensions of the groups they are constructed from; and then finally, in \cref{sec:poly}, we bound the asymptotic dimension of box spaces of groups of polynomial growth in terms of the growth rate and deduce \cref{thm:hirsch}.

\bigskip
\noindent\textsc{Acknowledgements.} The authors are grateful to Ana Khukhro, Alain Valette and Rufus Willett for helpful conversations; to Martin Finn-Sell for making available to us an as-yet unpublished version of \cite{FinnWu}; and to Alain Valette and two anonymous referees for careful readings of the manuscript and a number of helpful comments.

\section{Background}\label{sec:background}

In this section we collect together various results about box spaces and asymptotic dimension.
Given a group $G$ with a fixed finite generating set $S$ we write $B_G(x,R)$ for the ball of radius $R$ about the element $x\in G$ in the Cayley graph $\Cay(G,S)$.

\bigskip
\noindent\textsc{Coarse disjoint unions and box spaces.} Given a sequence $(X_n)_{n=1}^\infty$ of sets we write $\bigsqcup_{n}X_n$ for their disjoint union. If $(X_n,d_n)$ and $(\bigsqcup_{n}X_n,d)$ are all metric spaces then $\bigsqcup_{n}X_n$ is said to be a \emph{coarse disjoint union} of the $X_n$ if
\begin{enumerate}
\item for each $n$ the metric $d$ restricts to $d_n$ on $X_n$,
\item whenever $m\ne n$ the distance between $X_m$ and $X_n$ is at least $\diam X_m+\diam X_n$, and
\item for every $r$ the number of pairs $m,n$ with $d(X_m,X_n)<r$ is finite.
\end{enumerate}
Note that if $\diam X_i\to \infty$ then property (2) of this definition implies property (3).

\begin{remark*}
The exact distance between the components $X_n$ of a coarse disjoint union does not change the coarse equivalence class. Since asymptotic dimension is a coarse invariant \cite[Proposition 22]{asdimAlt}, neither does it change the asymptotic dimension.
\end{remark*}

In the introduction we defined a box space $\Box_{(N_i)}G$ of a finitely generated residually finite group $G$ with respect to a filtration $(N_n)$ of $G$ and a finite generating set $S$ of $G$ to be the disjoint union of the quotients $(G/N_n)_n$ with the metric that restricts to the Cayley-graph metric induced by $S$ on each $G/N_n$, and such that for $x\in G/N_m$ and $y\in G/N_n$ with $m\ne n$ the distance between $x$ and $y$ is precisely $\diam(G/N_m)+\diam(G/N_n)$. Note in particular that $\Box_{(N_i)}G$ is a coarse disjoint union of the $G/N_n$.

We also define the \emph{full box space} $\Box_f G$ of $G$, setting it to be the disjoint union of \emph{all} finite quotients of $G$ with the metric that restricts to the Cayley-graph metric induced by $S$ on each quotient, and such that for $x\in G/N$ and $y\in G/N'$ with $N\ne N'$ the distance between $x$ and $y$ is precisely $\diam(G/N)+\diam(G/N')$.

The following standard result says that the components of a box space locally `look like' the group.

\begin{proposition}\label{isoballs}
Let $G$ be a residually finite, finitely generated group and let $(N_i)_i$ be a filtration of $G$. Then there exists an increasing sequence $(i_k)_k$ such that for every $k\in\N$ the balls of radius $k$ of $G$ are isometric to the balls of radius $k$ of $G/N_i$, where $i\ge i_k$.
\end{proposition}
\begin{proof}
For a given $k$ and large enough $i$ we have $B_{G}(e,2k)\cap N_i=\{1\}$, which in turn implies that $B_{G}(e,k)$ is isometric to $B_{G/N_i}(e,k)$.
\end{proof}

\noindent\textsc{Asymptotic dimension.} There exist various equivalent alternative definitions of asymptotic dimension (see \cite[Theorem 19]{asdimAlt}, for example), and it will be useful for us to record one of these alternatives. First, given $r>0$ and a family $\mathcal U$ of subsets of a metric space $X$, we define the \emph{$r$-multiplicity} of $\mathcal U$ to be equal to
\[
\max_{x\in X}|\{U\in\mathcal U : U\cap B(x,r) \neq \varnothing\}|.
\]
Then we have the following result.
\begin{proposition}[{\cite[Theorem 19]{asdimAlt}}]\label{def:asdim}
Let $X$ be a metric space. Then $\asdim(X)\le n$ if and only if for every $R>0$ there exists $S>0$ and a covering $\mathcal{U}$ of $X$ such that $\mathcal{U}$ has $R$-multiplicity at most $n+1$ and $\diam(U)\le S$ for every $U\in\mathcal{U}$.
\end{proposition}

\begin{lemma}[{\cite[Finite Union Theorem]{Asdim}}]\label{lem:fut}
Let $X$ be a metric space and let $X_1,\ldots,X_n$ be a finite partition of  $X$. Then $\asdim X = \max\left(\asdim X_i\right)$.
\end{lemma}

Let $\mathcal U$ be a family of metric spaces. We say that $\mathcal U$ has asymptotic dimension at most $n$ \emph{uniformly}, and write $\asdim\mathcal U\le_\unif n$, if for every $R>0$ there exists $S>0$ such that $(R,S)\text{-}\dim X\le n$ for every $X\in\mathcal U$. This definition is particularly useful to us in light of the following result.
\begin{lemma}[{\cite[Theorem 1]{Asdim}}]\label{lem:unif.disjoint}
Let $X$ be a metric space, and let $\mathcal U$ be a family of subspaces that covers $X$. Suppose that $\asdim\mathcal U\le_\unif n$, and that for every $k\in\N$ there exists $F_k\subset X$ with $\asdim F_k\le n$ such that the family $\{Y\backslash F_k:Y\in\mathcal U\}$ is $k$-disjoint. Then $\asdim X\le n$.
\end{lemma}

\begin{corollary}\label{cor:split}
If $X$ is a coarse disjoint union of metric spaces $(X_n)_{n=1}^\infty$ then $\asdim(X_n)\le_\unif m$ if and only if $\asdim X\le m$.
\end{corollary}

We also record the following trivial fact as a lemma for ease of later reference.
\begin{lemma}\label{lem:unif.asdim}
Let $X$ be a metric space, and let $\mathcal{U}$ be a family of metric spaces each of which is isometric to a subspace of $X$. Then $\asdim\mathcal{U}\le_\unif\asdim X$.
\end{lemma}

The following result is presumably well known, although we could not find a reference.

\begin{lemma}\label{lem:cdu.balls}
Let $G$ be a finitely generated infinite group, and let $B$ be a coarse disjoint union of the balls $B_G(e,r)$ as $r$ ranges over the natural numbers. Then $\asdim B=\asdim G$.
\end{lemma}
\begin{proof}
The fact that $\asdim B\le\asdim G$ follows from \cref{cor:split} and \cref{lem:unif.asdim}.

To prove that $\asdim G\le\asdim B$ it suffices to show that $(R,S)\text{-}\dim G\le(R,S)\text{-}\dim B$ for every $R,S>0$. To that end, fix $R,S>0$ and suppose that $(R,S)\text{-}\dim B=n\in\Z$, so that there exist $R$-disjoint families $\mathcal{U}_0,\mathcal{U}_1,\ldots,\mathcal{U}_n$ of subsets of $B$ that cover $B$ such that $\diam(U) \le S$ for every $U \in \mathcal{U}_j$ and $j\in\{0,1,...,n\}$.

We partition $G$ into sets $U_0,\ldots,U_n$ as follows. First, enumerate the elements of $G$ as $x_1,x_2,x_3,\ldots$ in such a way that $|x_m|$ is non-decreasing. We will specify for the $x_m$ in turn which set $U_j$ will contain $x_m$. Note that for each $m$ and each $r\ge|x_m|$ there is a copy of $x_m$ lying in the component $B_G(e,r)$ of $B$.

There exists $i_1$ and an infinite sequence $r_{1,1}<r_{1,2}<\ldots$ such that for each $r_{1,j}$ the copy of $x_1$ in the component $B_G(e,r_{1,j})$ lies in a set belonging to $\mathcal U_{i_1}$. We declare that $x_1\in U_{i_1}$. Similarly, there exists $i_2$ and an infinite subsequence $r_{2,1}<r_{2,2}<\ldots$ of $r_{1,1}<r_{1,2}<\ldots$ such that for each $r_{2,j}$ the copy of $x_2$ in the component $B_G(e,r_{2,j})$ lies in a set belonging to $\mathcal U_{i_2}$. We declare that $x_2\in U_{i_2}$. Continuing in this way, for each $m$ in turn there exists $i_m$ and an infinite subsequence $r_{m,1}<r_{m,2}<\ldots$ of $r_{m-1,1}<r_{m-1,2}<\ldots$ such that for each $r_{m,j}$ the copy of $x_m$ in the component $B_G(e,r_{m,j})$ lies in a set belonging to $\mathcal U_{i_m}$. We declare that $x_m\in U_{i_m}$.

It follows from the definition of the $\mathcal U_i$ that each $U_i$ can be partitioned into subsets of diameter at most $S$ that are $R$-disjoint, and this completes the proof.
\end{proof}

\section{The asymptotic dimension of arbitrary box spaces}\label{sec:arb}
Yamauchi \cite[Theorem 1.3]{yamauchi} shows that a coarse disjoint union of a sequence of graphs with girth tending to infinity has asymptotic dimension either infinite or at most $1$. The following is an adaptation of his argument. We are grateful to Rufus Willett for pointing it out to us.

\begin{proposition}\label{boxSplit}
Let $G$ be an infinite, residually finite, finitely generated group and let $N_n$ be a filtration. Then $\asdim\left(\Box_{(N_n)}G\right)$ is either infinite or equal to $\asdim G$. 
\end{proposition}
\begin{proof}
We may assume that $\asdim\left(\Box_{(N_n)}G\right)<\infty$, and so $\asdim\left(\Box_{(N_n)}G\right)=m$ for some $m\in\N$. We first prove that $\asdim\left(\Box_{(N_n)}G\right)\le\asdim G$.

By definition, for every $k\in\N$ there exists $S_k\in\N$ such that $(k,S_k)\text{-}\dim\left(\Box_{(N_n)}G\right)\le m$. For every such $k$ there therefore exist $k$-disjoint families $\mathcal{U}_0^k,\mathcal{U}_1^k,\ldots,\mathcal{U}_m^k$ of subsets of $\Box_{(N_n)}G$, such that $\diam(U)\le S_k$ for every $U\in\mathcal U_j^k$ and such that $\bigcup_{j=1}^m\mathcal U_j^k$ covers $\Box_{(N_n)}G$.
By \cref{isoballs} we can take a sequence $i_k\to\infty$ such that for every $i\ge i_k$ the balls of radius $\max\{k,S_k\}$ in $G/N_i$ are isometric to balls of radius $\max\{k,S_k\}$ in $G$. Without loss of generality we may assume that $(i_k)_k$ is a non-decreasing sequence.

If $U\in\mathcal U_j^k$ satisfies $U\subset G/N_i$ for some $i$ then $U$ is contained in a ball of radius at most $S_k$ inside $G/N_i$, and if this $i$ is at least $i_k$ then $U$ is isometric to a subspace of $G$. \cref{lem:unif.asdim} therefore implies that
\begin{equation}\label{eq:asdim.unif}
\asdim\left(\bigcup_{j,k}\{U\in\mathcal U_j^k:U\subset G/N_i\text{ for some }i\ge i_k\}\right)\le_\unif\asdim G.
\end{equation}

Now if $i\ge i_k$ then $G/N_i$ is at a distance greater than $S_k$ from its complement in $\Box_{(N_n)}G$, and so if $U\in\mathcal{U}_j^k$ intersects $G/N_i$ non-trivially then in fact $U\subset G/N_i$. This implies that for every $i\ge i_k$ the set $G/N_i$ is covered by the sets $U\in\mathcal U_j^k$ with $U\subset G/N_i$ and $j\in\{0,\ldots,m\}$. This means that, defining families $\mathcal{Y}_j^k$ of subsets of $\Box_{(N_n)}G$ for $j\in\{0,\ldots,m\}$ and $k\in\N$ via
\[
\mathcal{Y}_j^k=\{U\in\mathcal{U}_j^k : U\subset G/N_i\text{ for some }i\in[i_k,i_{k+1})\},
\]
and then defining families $\mathcal{Y}_j^{\ge k}$ via
\[
\mathcal{Y}_j^{\ge k}=\bigcup_{k'\ge k}\mathcal{Y}_j^{k'},
\]
for each $k$ the set $\bigcup_{i=i_k}^\infty G/N_i$ is covered by the families $\mathcal{Y}_j^{\ge k}$ with $j\in\{0,\ldots,m\}$. In particular, setting $F=\bigcup_{i=1}^{i_1-1}G/N_i$ and $X_j=\bigcup_{Y\in\mathcal{Y}_j^{\ge 1}}Y$ we have
\begin{equation}\label{eq:box.covering}
\Box_{(N_n)}G=F\cup\bigcup_{j=0}^mX_j.
\end{equation}

Note that for each $j$ we have $\asdim\mathcal{Y}_j^{\ge 1}\le_\unif\asdim G$ by \eqref{eq:asdim.unif}. Note also that each family $\mathcal{Y}_j^{\ge k}$ is $k$-disjoint by the definitions of $\mathcal{U}_j^k$ and $i_k$. Finally, note that if we write $F_k=\bigcup_{i=1}^{i_k-1}G/N_i$ then we have $\mathcal{Y}_j^{\ge k}=\{ Y\backslash F_k  :   Y\in\mathcal{Y}_j^{\ge1}\}$. Since $F_k$ is finite, and hence of asymptotic dimension $0$, it therefore follows from \cref{lem:unif.disjoint} that $\asdim X_j\le\asdim G$, and then from \eqref{eq:box.covering} and \cref{lem:fut} that
\[
\asdim\left(\Box_{(N_n)}G\right)\le\asdim G,
\]
as desired.

Conversely, since $N_n$ is a filtration there is a subspace $B$ of $\Box_{(N_n)}G$ that is isometric to a coarse disjoint union of the balls $B_G(e,R)$ as $R$ ranges over the natural numbers. It follows from \cref{lem:cdu.balls} that $\asdim B=\asdim G$, and since $B$ is a subspace of $\Box_{(N_n)}G$ this implies that
\[
\asdim\left(\Box_{(N_n)}G\right)\ge\asdim G,
\]
which completes the proof.
\end{proof}

\section{Coarse disjoint unions of groups of polynomial growth}\label{sec:poly}
A group G with fixed finite generating set is said to have \emph{polynomial growth of degree $d$} if there exists $C>0$ such that $|B_G(e,r)|\le Cr^d$ for every $r\ge1$. It is well known and easy to check that this notion does not depend on the choice of finite generating set.

We say that a family $(G_\alpha)_{\alpha\in A}$ of groups with fixed generating sets $S_\alpha$ has \emph{uniform polynomial growth of degree at most $d$} if there exists $C>0$ such that $|B_{G_\alpha}(e,r)|\le Cr^d$ for every $r\ge1$ for every $\alpha\in A$.

\begin{proposition}\label{finite}
Let $G_n$ be a sequence of finite groups with generating sets $S_n$, and suppose that $(\Cay(G_n,S_n))_{n=1}^\infty$ has uniform polynomial growth of degree at most $d$. Let $X$ be a coarse disjoint union of the $\Cay(G_n,S_n)$. Then $\asdim X\le4^d$.
\end{proposition}

\begin{proof}
We start with the standard observation that a polynomial growth bound implies a so-called \emph{doubling} condition, as used by Gromov \cite{gromov} in his proof of his polynomial-growth theorem, for example. Specifically, let $K=4^d+1$, let $R>0$, and take $S_0=4^{m+1}R$ with $m$ such that $(K/4^d)^m\ge CR^d$. Then we claim that for every $n$ there exists an $R_n$ such that $R\le R_n\le \frac{S_0}{4}$ and $|B_{G_n}(4R_n)|\le K|B_{G_n}(R_n)|$. Indeed, if this were not the case then $K^i|B_{G_n}(R)|<|B_{G_n}(4^iR)|\le C4^{id}R^d$ for every $i$ with $4^iR\le S_0$, and so setting $i=m$ would imply that $C4^{md}R^d > K^m|B_{G_n}(R)|\ge C4^{md}R^d|B_{G_n}(R)|$, contradicting the growth assumption. 

Following Ruzsa \cite{ruzsa}, for every $n\in \N$ with $\diam(G_n)>R$ we take $X_n$ maximal in $G_n$ such that $B_{G_n}(x,R_n)$ and $B_{G_n}(y,R_n)$ are disjoint for every $x$ and $y$ in $X_n$. We then define $F_R=\bigcup_{n\,:\,\diam(G_n)\le R}G_n$, take
\[
\mathcal{U} = \{F_R\} \cup \bigcup_{n\,:\,\diam(G_n)> R}\{B_{G_n}(x,2R_n):x\in X_n\},
\]
and set $S=\max\{S_0,\diam F_R\}$.

First we show that $\mathcal{U}$ is a covering of $\bigsqcup_{n}G_n$. Let $z\in G_m$ for some $m$. If $\diam(G_m)\le R$, then $z\in F_R$. If $\diam(G_m)> R$, then as $X_m$ is maximal, there exists an $x\in X_m$ such that $B_{G_m}(z,R_m) \cap B_{G_m}(x,R_m)$ is non-empty, so $z\in B_{G_m}(x,2R_m)$.

Next we note that $\diam(U)\le S$ for every $U\in\mathcal{U}$. For $U=F_R$ this is true by definition of $S$. On the other hand, for $U\in\mathcal{U}$ with $U\ne F_R$ we have $U\subset G_m$ for some $m$ and $\diam(U)=4R_m\le S_0\le S$.

Finally, we show that $\mathcal{U}$ has $R$-multiplicity at most $K$. The $R$-multiplicity of $\mathcal{U}$ in $G_n$ with $\diam(G_n)\le R$ is $1\le K$, so take $z\in G_m$ with $m$ such that $\diam(G_m)\ge R$.
Now for every $B_{G_m}(x,2R_m)\in\mathcal{U}$ which has an element at a distance at most $R$ to $z$, we have that $x\in B_{G_m}(z,2R_m+R)\subset B_{G_m}(z,3R_m)$.
Now consider $B_{G_m}(z,3R_m)\cap X_m$. As $B_{G_m}(x,R_m)$ and $B_{G_m}(y,R_m)$ are disjoint for any $x$ and $y$ in $X_m$, we have that $|B_{G_m}(z,3R_m)\cap X_m|\le \frac{|B_{G_m}(z,4R_m)|}{|B_{G_m}(z,R_m)|} \le K$.
Therefore the $R$-multiplicity of $\mathcal{U}$ is at most $K=4^d+1$, and so $\asdim\bigsqcup_n G_n \le 4^d$ by \cref{def:asdim}.
\end{proof}

\begin{corollary}\label{cor:full.box}
Let $G$ be a finitely generated virtually nilpotent group. Then $\Box_fG$ has finite asymptotic dimension. 
\end{corollary}
\begin{proof}
As $G$ is virtually nilpotent there exist constants $k$ and $C$ such that $|B_G(e,r)|\le Cr^k$ for every $r\ge1$. This also means that $|B_{G/N}(e,r)|\le Cr^k$ for every $r\ge1$ for every $N\lhd G$, and so the result follows from \cref{finite}.
\end{proof}

\begin{proof}[Proof of \cref{thm:hirsch}]
Since $\Box_{(N_n)}G$ is a subspace of $\Box_fG$ we have $\asdim \Box_{(N_n)}G \le \asdim \Box_fG<\infty$ by \cref{cor:full.box}. \cref{boxSplit} therefore implies that $\asdim \Box_{(N_n)}G=\asdim(G)$. Since $G$ is virtually polycyclic, the result therefore follows from \eqref{eq:hirsch}. 
\end{proof}

\bibliographystyle{alpha}
\bibliography{bib.bib}

\end{document}